\documentclass[a4paper,11 pt,reqno]{amsart}
\usepackage{a4,amsbsy,amscd,amssymb,amstext,amsmath,latexsym,oldgerm,enumerate,verbatim,xy}

\makeatletter                                       %De subsecties in de inhoudsopgave moeten
\def\l@subsection{\@tocline{2}{0pt}{2.5pc}{2.5pc}{}}%iets opgeschoven worden naar rechts
\makeatother                                        %een eventuele tweede regel moet gewoon
\makeatletter                                                  %Hoofdstukken moeten op de eerst
\def\chapter{\clearpage\thispagestyle{plain}\global\@topnum\z@ %beschikbare pagina beginnen
\@afterindenttrue \secdef\@chapter\@schapter}
\makeatother

\newtheorem{thmgl} {Theorem}    %globaal genummerd

\newtheorem{lemgl} {Lemma}
\newtheorem{thmnn}{Theorem}                %ongenummerd

\newtheorem{propnn}{Proposition}

\newtheorem{cornn}{Corollary}

\theoremstyle{definition}

\newtheorem{remsnn}{Remarks}

\newcommand{\mf}{\mathfrak}
\newcommand{\mc}{\mathcal}
\newcommand{\mb}{\mathbb}

\newcommand{\nts}{\negthinspace}     %handig

\newcommand{\dt}{{\scriptscriptstyle\bullet}} %for the dot action

\newcommand{\sm}{\setminus}         %verzamelingen

\newcommand{\ot}{\otimes}           %vectorruimten en modules

\newcommand{\Hom}{{\rm Hom}}        %Algebra algemeen
\newcommand{\End}{{\rm End}}

\newcommand{\Ext}{{\rm Ext}}
\newcommand{\Tor}{{\rm Tor}}

\newcommand{\Maxspec}{{\rm Maxspec}}
\newcommand{\rank}{{\rm rank}}

\newcommand{\gr}{{\rm gr}}   %gegradeerde algebra v/e gefilterde algebra

\newcommand{\g}{\mf{g}}
\let\ttie\t
\newcommand{\tie}[1]{{\let\t\ttie \ttie#1}}%\t requires a special treatment, because
\renewcommand{\t}{\mf{t}}  %it is defined recursively. This trick is due to Uwe L\"uck

\renewcommand{\u}{\mf{u}}

\renewcommand{\b}{\mf{b}}
\newcommand{\ad}{{\rm ad}}               %Lie Algebra's
\newcommand{\gl}{\mf{gl}}
\newcommand{\spl}{\mf{sl}}

\newcommand{\Ad}{{\rm Ad}}              %Algebraische groepen
\newcommand{\Lie}{{\rm Lie}}
\newcommand{\GL}{{\rm GL}}

\newcommand{\SL}{{\rm SL}}

\xyoption{all}

\begin{document}
\title{The Zassenhaus variety of a reductive Lie algebra in positive characteristic}

%\title{On the centre of the universal enveloping algebra in positive characteristic}

\begin{abstract}
Let $\g$ be the Lie algebra of a connected reductive group $G$ over an algebraically closed field $k$ of characteristic $p>0$. Let $Z$ be the centre of the universal enveloping algebra $U=U(\g)$ of $\g$. Its maximal spectrum is called the Zassenhaus variety of $\g$. We show that, under certain mild assumptions on $G$, the field of fractions ${\rm Frac}(Z)$ of $Z$ is $G$-equivariantly isomorphic to the function field of the dual space $\g^*$ with twisted $G$-action. In particular ${\rm Frac}(Z)$ is rational. This confirms a conjecture J.~Alev. Furthermore we show that $Z$ is a unique factorisation domain, confirming a conjecture of A.~Braun and C.~Hajarnavis. Recently, A.~Premet used the above result about ${\rm Frac}(Z)$, a result of Colliot-Thelene, Kunyavskii, Popov and Reichstein and reduction mod $p$ arguments to show that the Gelfand-Kirillov conjecture cannot hold for simple complex Lie algebras that are not of type $A$, $C$ or $G_2$.
\end{abstract}

\author[R.\ H.\ Tange]{Rudolf Tange}


\maketitle

\pagestyle{myheadings}
\markright{\uppercase{The Zassenhaus variety of a reductive Lie algebra}}

\section*{Introduction}
%Emphasize the relevance of Z for the representation theory of \g. Open problems?
Throughout $k$ is an algebraically closed field of characteristic $p>0$ and $G$ is a connected reductive group over $k$. We denote the Lie algebra of $G$ by $\g$, the universal enveloping algebra of $\g$ by $U(\g)$ or just $U$ and the centre of $U$ by $Z$. Since $U$ has no zerodivisors, $Z$ is a domain. The adjoint actions of $G$ and $\g$ on $\g$ induce actions of $G$ and $\g$ on $U$ by automorphisms and derivations respectively. The same applies to the symmetric algebra $S(\g)$ on $\g$. The action of $\g$ on $U$ is given by $x\cdot u=xu-ux$ for $x\in\g$ and $u\in U$. It follows that $Z=U^\g:=\{u\in U\,|\,x\cdot u=0\text{\ for all\ }x\in\g\}$.

As in characteristic zero we have, under certain mild conditions on $G$, that the invariant algebra $U^G:=\{u\in U\,|\,g\cdot u=u\text{\ for all\ }g\in G\}$ is a polynomial algebra in $\rank(G)$ variables. However, unlike in characteristic $0$, we have that the inclusion $U^G\subseteq U^\g=Z$ is proper. In fact $Z$ has Krull dimension $\dim(G)=\dim(\g)$ and $U$ is a finite module over $Z$. As a consequence all irreducible $\g$-modules are finite dimensional with an upper bound only depending on $\g$ for their dimension. The maximal spectrum of $Z$ is called the {\it Zassenhaus variety} of $\g$ and we denote it by $\mc Z$. It is a normal irreducible variety, but it is not smooth. Denote the set of isomorphism classes if irreducible $\g$-modules by ${\rm Irr}(\g)$. Then Schur's Lemma gives us a map $\varphi:{\rm Irr}(\g)\to \mc Z$. This map is surjective with finite fibres. Zassenhaus \cite{Z} showed, in a much more general setting, that the points of $\mc Z$ of which the fibre contains an irreducible module of maximal dimension form a special open set $\mc O$ of $\mc Z$ and that the fibres of these points are in fact singletons. Brown and Goodearl \cite{BG} proved that $\mc O$ coincides with the smooth locus of $\mc Z$. A result of M\"uller \cite{Mu} implies that the fibers of $\varphi$ are the blocks of $U$, see also \cite{BG2}. From this and Veldkamp's description of the centre one deduces Humphreys' conjecture on the block structure of $U$, see \cite{BrGo}.

The strong link with the representation theory of $\g$ as well as the fact that $U$ is an important example in ring theory are important motivations for the study of the Zassenhaus variety.

In this paper we want to consider the following two conjectures concerning the centre $Z$:
\begin{enumerate}[(A)]
\item The field of fractions ${\rm Frac}(Z)$ of $Z$ is rational, that is, purely transcendental over $k$.
\item The centre $Z$ is a unique factorisation domain.
\end{enumerate}
Conjecture (A) was posed by J.~Alev and conjecture (B) was posed by A.~Braun and C.~Hajarnavis. Since ${\rm Frac}(Z)$ is the function field of $\mc Z$, conjecture (A) states that $\mc Z$ is a rational variety. It can be considered as a commutative version of the Gelfand-Kirillov conjecture which states that ${\rm Frac}(U)$ is isomorphic to the field of fractions of a Weyl algebra $D_n(L)$, where $L$ is a purely transcendental extension of $k$. In this paper we will prove both conjectures under certain mild hypotheses on $G$. These conjectures were proved in type $A$ by A.~Premet and the author. The proof for the general case given in the present paper is quite different, since the arguments for type $A$ don't work in general.

The paper is organised as follows. In Section~\ref{s.prelim} we recall some basic results from the literature that will be needed in the next two sections. This section is rather long, since our approach requires a detailed understanding of these results. Subsection~\ref{ss.standardhyp} contains the hypotheses that we make on $G$. In Section~\ref{s.rationality} we prove conjecture (A). From the Galois covering with group $W$ of the variety $\g^*_{\rm rss}$ of regular semisimple elements of $\g^*$  (see Subsection~\ref{ss.regsselts}) we construct a Galois covering with group $W$ of a special open subset of $\mc Z$. This covering turns out to be isomorphic to a special open subset of the covering of $\g^*_{\rm rss}$. The idea to work with the Galois covering of $\g^*_{\rm rss}$ was  suggested by the description of $\mc Z$ at the beginning of \cite[Sect.~3]{Kry}. There it is made clear that ``the diagonal form" of the semisimple part of the $p$-character is in a natural way incorporated in the equations of the Zassenhaus variety. In Section~\ref{s.ufd} we prove conjecture (B). The proof relies strongly on that of conjecture (A). We obtain as a corollary that every height one prime of $U$ is generated by a single central element and, combining with Braun's work \cite{Braun1}, that $U$ is a Calabi-Yau algebra over every polynomial subalgebra of $Z$ over which $Z$ is module finite.

\section{Preliminaries}\label{s.prelim}
Let $\mb F_p$ be the prime subfield of $k$. Fix an ${\mb F}_p$-structure on $G$ such that $G$ is split over ${\mb F}_p$ and let $T$ be a maximal torus which is defined and split over ${\mb F}_p$, see \cite[AG.11,V.18]{Bo}. We denote the Lie algebra of $T$ by $\t$, the character group of $T$ by $X(T)$ and the Weyl group of $G$ relative to $T$ by $W$. Let $\Sigma\subseteq X(T)$ be the set of roots of $G$ relative to $T$. To each root $\alpha\in\Sigma$ there is associated a coroot $\alpha^\vee: k^\times\to G$, see \cite[Ch.~7]{Spr}. We put $h_\alpha=d\alpha^\vee(1)\in\t^*$, where $d\alpha^\vee:k\to\t$ is the differential of $\alpha^\vee$. We fix a system of positive roots $\{\alpha\in\Sigma\,|\,\alpha>0\}$; this also means that we have chosen a set of simple roots.

The Lie algebra $\g$ is a restricted Lie algebra, we denote its $p$-mapping by $x\mapsto x^{[p]}:\g\to\g$. For any Lie algebra of an algebraic group and in fact for any restricted Lie algebra one can define the notions semisimple and nilpotent for the elements of the Lie algebra; see \cite{Bo} or \cite{Str-Far}. The Lie algebra $\t$ of $T$ is a restricted subalgebra of $\g$. It consists of semisimple elements and its $p$-mapping is $p$-semilinear, since $\t$ is abelian. The $\mb F_p$-structure on $G$ induces one on $\g$: $\g=k\ot_{\mb F_p}\g(\mb F_p)$. The adjoint action of $G$ on $\g$ and the $p$-mapping of $\g$ are defined over $\mb F_p$ and $\t$ is an $\mb F_p$-defined restricted subalgebra of $\g$. An element $x$ of $\g$ is called {\it toral} if $x^{[p]}=x$. The toral elements of $\t$ are the elements of $\t(\mb F_p)$.

For a variety $V$ we denote the algebra of (everywhere defined) regular functions on $V$ by $k[V]$. If $V$ is a finite dimensional vector space over $k$, then we have a canonical isomorphism $k[V]\cong S(V^*)$, where $S(V^*)$ denotes the symmetric algebra on $V^*$. Throughout this paper we will identify $S(\g)$ and $k[\g^*]$ and also $S(\t)$ and $k[\t^*]$ by means of this isomorphism.

In Subsections~\ref{ss.filtration}-\ref{ss.veldkamp} below we follow \cite[Sect.~2,3]{PrT} which is based on the ideas in \cite{Ve}.

\subsection{The Frobenius twist and the $p$-centre}\label{ss.pcentre}
For more precise and general definitions concerning the Frobenius twist we refer to \cite[I.9.1,2,10; II.3.16]{Jan1}.
For a vector space $V$ over $k$ the {\it Frobenius twist} $V^{(1)}$ of $V$ is defined as the vector space over $k$ with the same additive group as $V$ and with scalar multiplication given by $a\cdot x=a^{1/p}\,x$. Note that the linear functionals and the polynomial functions on $V^{(1)}$ are the $p$-th powers of those of $V$. In fact we have an isomorphism $V^{*(1)}\stackrel{\sim}{\to}V^{(1)*}$ of vector spaces given by $f\mapsto f^p$. The identity map $V\to V^{(1)}$ is a morphism of varieties, we call it the {\it Frobenius morphism}. If $V$ has an $\mb F_p$-structure $V\cong k\ot_{\mb F_p}V(\mb F_p)$, then we obtain an isomorphism $V^{(1)}\cong V$ which is given by the identity on $V(\mb F_p)$. The Frobenius morphism yields then a Frobenius endomorphism of $V$. The Frobenius endomorphism raises the coordinates with respect to an $\mb F_p$ basis of $V(\mb F_p)$ to the $p$-th power. The Frobenius twist of a $k$-algebra is defined similarly (only the scalar multiplication is modified). If $V$ is a finite dimensional vector space over $k$, then we have an isomorphism $\End_k(V^{(1)})\cong\End_k(V)^{(1)}$ of algebras over $k$ which is given by the identity map.

More generally, one can define the Frobenius twist $X^{(1)}$ of a variety $X$ over $k$ and the Frobenius morphism $Fr:X\to X^{(1)}$. If $X$ has an $\mb F_p$-structure, then we obtain an isomorphism $X^{(1)}\cong X$ and a Frobenius endomorphism of $X$. If $X$ is a closed subvariety of affine space $\mb A^n(k)=k^n$, then $X^{(1)}$ can be identified with $Fr(X)$, where $Fr$ is the Frobenius endomorphism of $\mb A^n(k)$; if $X$ is $\mb F_p$-defined, then its Frobenius endomorphism is induced by that of $\mb A^n(k)$. If $G'$ is a linear algebraic group, then so is $G'^{(1)}$ and the Frobenius morphism is a morphism of algebraic groups. If $G'$ has an $\mb F_p$-structure, then the Frobenius endomorphism is an endomorphism of algebraic groups. If $V$ is a finite dimensional vector space over $k$, then we have an isomorphism $\GL(V^{(1)})\cong\GL(V)^{(1)}$ of algebraic groups.

If $V$ is a finite dimensional rational $G$-module, then we can turn $V^{(1)}$ into a $G$-module by composing $G\to\GL(V)$ with the Frobenius morphism $\GL(V)\to\GL(V^{(1)})$. Note that, when passing from the $G$-module $V$ to the $G$-module $V^{(1)}$, the weights of $T$ get multiplied by $p$. We do have $\t\cong\t^{(1)}$ as $W$-modules, since $\t(\mb F_p)\subseteq\t$ is $W$-stable. The above construction extends to infinite dimensional rational $G$-modules. We can also compose $G\to\GL(V)$ with the Frobenius endomorphism of $G$; we denote the resulting $G$-module by $V^{[1]}$. If $V$ has an $\mb F_p$-structure and the representation is defined over $\mb F_p$, then we get an isomorphism $V^{[1]}\cong V^{(1)}$.

The {\it $p$-centre} $Z_p$ of $U$ is defined as the subalgebra of $U$ generated by all elements $x^p-x^{[p]}$ with $x\in \g$. It is well-known (and easily seen) that $Z_p\subseteq Z$ is a polynomial algebra in $x_i^p-x_i^{[p]}$ where $\{x_i\}$ is any basis of $\g$. Following \cite{Kry} we define $\eta:S(\g)^{(1)}\to Z_p$ by setting $\eta(x)=x^p-x^{[p]}$ for all $x\in\g$. This is a $G$-equivariant algebra isomorphism. Considering $\eta$ as a homomorphism from $S(\g)^{(1)}$ to $U$, we have that $\eta(S(\t)^{(1)})\subseteq U(\t)=S(\t)$.

\subsection{The standard hypotheses for reductive groups}\label{ss.standardhyp}
Throughout this paper we assume that $G$ satisfies the following standard hypotheses (see \cite[6.3]{Jan2}).
\begin{enumerate}[(H1)]
\item The derived group $DG$ of $G$ is simply connected.
\item $p$ is good for $G$.
\item There exists a $G$-invariant non-degenerate bilinear form on $\g$.
\end{enumerate}

Recall that a prime is called {\it good} for $G$ if it is good for the root system $\Sigma$ and that a prime is called good for a root system if it is good for each irreducible component. The bad (i.e. not good) primes for the irreducible root systems are as follows: none for type $A_n$; $2$ for types $B_n$, $C_n$, $D_n$; $2$ and $3$ for types $E_6$, $E_7$, $F_4$ and $G_2$; $2$, $3$ and $5$ for type $E_8$.

Hypothesis (H3) says that $\g\cong\g^*$ as $G$-modules. Standard arguments show that a nondegenerate $G$-invariant bilinear form on $\g$ restricts to a nondegenerate $W$-invariant bilinear form on $\t$. So we also have $\t\cong\t^*$ as $W$-modules. In this paper we will not use these isomorphisms to replace $\g^*$ by $\g$ and $\t^*$ by $\t$, because we want to keep all (iso)morphisms as natural as possible.

Hypothesis (H1) implies that the $h_\alpha$ with $\alpha$ simple are linearly independent. Furthermore, it implies that $kh_\alpha=[\u_{-\alpha},\u_\alpha]\ne0$ for all $\alpha\in\Sigma$. Using (H3) we deduce that the $G$-module isomorphism $\g\cong\g^*$ given by the nondegenerate $G$-invariant bilinear form on $\g$ maps $h_\alpha$ to a nonzero multiple of $d\alpha$, see \cite[11.2]{Jan2}. So the elements $d\alpha$ with $\alpha$ simple are also linearly independent.

Note that $\GL_n$ always satisfies the hypotheses (H1)-(H3), but that $\SL_n$ only satisfies them when $p\nmid n$.

\subsection{Regular semisimple elements and the restriction theorem for $\g^*$.}\label{ss.regsselts}
We recall from \cite[Sect.~3]{KW} the notions of semisimple and nilpotent elements and the Jordan decomposition for elements of $\g^*$. Let $\u_\alpha$ be the root space corresponding to the root $\alpha$, let $\u^+$ and $\u$ be the sum of the positive and negative root spaces respectively and put $\b=\u+\t$. Embed $\t^*$, $\u^*$ and $\b^*$ in $\g^*$ by requiring the functionals to be zero on $\u+\u^+$, $\t+\u^+$ and $\u^+$ respectively. Then $\b^*=\t^*\oplus\u^*$. An element $\chi\in\g^*$ is called {\it semisimple} resp. {\it nilpotent} if some $G$-conjugate of $\chi$ lies in $\t^*$ resp. $\u^*$. The union of the $G$-conjugates of $\b^*$ is $\g^*$. Now $\chi=\chi_s+\chi_n$ is called a {\it Jordan decomposition} of $\chi$ if there exists a $g\in G$ such that
\begin{align*}
&g\cdot\chi_s\in\t^*, g\cdot\chi_n\in\u^*\\
&\text{and for any positive root }\alpha,\\
&g\cdot\chi_s(h_\alpha)\ne 0\text{ implies }g\cdot\chi_n(\u_\alpha)=0.
\end{align*}
Every element of $\g^*$ has a unique Jordan decomposition. Let $B^+$ be the Borel subgroup containing $T$ with $\Lie(B^+)=\t+\u^+$. Then $\b^*$ and $\u^*$ are $B^+$-stable. Furthermore, if $\chi\in\b^*$, then all $B^+$-conjugates of $\chi$ have the same $\t^*$-part. Now let $\chi\in\b^*$. Then there exists a $b\in B^+$ such that $b\cdot\chi_s$ and $b\cdot\chi_n$ have the displayed properties, in particular, $\chi_s\in\b^*$ and $\chi_n\in\u^*$; see \cite[3.8]{KW}.

An element $\chi\in\t^*$ is called {\it regular} if $\chi(h_\alpha)\ne0$ for all $\alpha$. We denote set the regular elements of $\t^*$ by $\t^*_{\rm reg}$. The union of the $G$-conjugates of $\t^*_{\rm reg}$ in $\g^*$ is the set $\g^*_{\rm rss}$ of regular semisimple elements of $\g^*$, it is an open dense subset of $\g^*$. Note that for $\chi\in\g^*$, $\chi_s$ regular semisimple implies $\chi=\chi_s$. Let $\Phi:S(\g)=k[\g^*]\to k[\t^*]=S(\t)$ be the homomorphism that restricts functions to $\t^*$. If we identify $S(\g)$ with $S(\u)\ot S(\t)\ot S(\u^+)$, then $\Phi(x\ot h\ot y)=x^0hy^0$ where $f^0$ denotes the zero degree part of $f\in S(\g)$. The homomorphism $\Phi$ restricts to an isomorphism
$$\Phi:S(\g)^G\stackrel{\sim}{\to}S(\t)^W;$$
see \cite[Thm.~4]{KW}. All that is needed for the above results is that $h_\alpha\ne0$ for all $\alpha\in\Sigma$.
%this is implied by: $DG$ does not have an adjoint group of type $B$ as a simple factor in case $p=2$.
Clearly this is implied by (H1). One can, of course, also define semisimple and nilpotent elements and the Jordan decomposition in $\g^*$ using the $G$-module isomorphism $\g\cong\g^*$ given by (H3) and the corresponding classical notions for $\g$, see \cite{Jan2}.

The above results and definitions for semisimple and nilpotent elements also apply to $\g^{*(1)}$, since the underlying additive group of $\g^{*(1)}$ and the $G$-action on it is the same as that of $\g^*$. Using the $G$-module isomorphism $\g^{*(1)}\cong\g^{(1)*}$, this also applies to $\g^{(1)*}$. Note that $\Phi$ induces isomorphisms $S(\g)^{(1)G}\stackrel{\sim}{\to}S(\t)^{(1)W}$ and $(S(\g)^p)^G\stackrel{\sim}{\to}(S(\t)^p)^W$, so we get restriction theorems for $\g^{(1)*}$ and for $\g^{*(1)}$.

We recall the description of the variety of regular semisimple elements of $\g$ in \cite[II.3.17' proof]{Spr-St} (see also \cite[Sect.~2]{D}). Since we need this result for $\g^*$, we only give the version for $\g^*$. We define the action of $W$ on $G/T\times\t^*_{\rm reg}$ by $w(gT,\lambda)=(gw^{-1}T,w(\lambda))$. Note that this action commutes with the action of $G$ that comes from the left multiplication on the first factor.

We refer to \cite{St} or \cite{Dem} for the definition of torsion primes. Since $DG$ is simply connected, the torsion primes are those of $\Sigma$, see e.g. \cite[Cor~2.6]{St}. Since $p$ is good for $\Sigma$ by (H1) it is not a torsion prime of $\Sigma$ and therefore also not a torsion prime of $G$. By \cite[Lem.~3.7, Thm.~3.14]{St} we now have that the stabiliser in $W$ of an element of $\t$ is a reflection subgroup of $W$. By (H3) the same holds for the elements of $\t^*$. In particular, the elements of $\t^*_{\rm reg}$ have trivial stabiliser in $W$.

We have a $G$-equivariant surjective morphism $G/T\times\t^*_{\rm reg}\to\g^*_{\rm rss}$ given by $(gT,\lambda)\mapsto g\cdot\lambda$. By the above the fibres of this morphism are the $W$-orbits. Furthermore, one easily checks that its differentials are surjective.
%Idea: The map (g,x)\mapsto Ad(g)(x) is (g,x)\mapsto (Ad(g),x) composed with (A,x)\mapsto A(x) which is a bilinear map.
So, by \cite[Prop.~II.6.6, Thm.~AG.17.3]{Bo}, we obtain a $G$-equivariant isomorphism
$$(G/T\times\t^*_{\rm reg})/W\stackrel{\sim}{\to}\g^*_{\rm rss}\ .$$

\begin{comment}
\begin{lemgl}
If two elements of $\t$ are $G$-conjugate, then they are $W$-conjugate.
\end{lemgl}
\begin{proof}
Let $x,y\in\t$ and assume that $\Ad(g)(x)=y$ for some $g\in G$. Then $T$ and $gTg^{-1}$ centralise $y$. So, by the conjugacy of maximal tori, there exists an $h\in C_G(y)$ such that $gTg^{-1}=hTh^{-1}$. Then $h^{-1}g\in N_G(T)$ and $\Ad(h^{-1}g)(x)=\Ad(h^{-1})(y)=y$.
\end{proof}
\end{comment}

\subsection{Filtrations}\label{ss.filtration}
Let $A$ be an associative ring with an
ascending filtration $(A_i)_{i\in\mb{Z}}$. If $I$ is a two sided ideal of $A$, then the abelian group $I$ and the ring $A/I$ inherit an ascending filtration from A and we have an embedding $\gr(I)\hookrightarrow\gr(A)$ of graded abelian groups. If we identify $\gr(I)$ with a graded subgroup of the graded additive group $\gr(A)$ by means of this embedding, then $\gr(I)$ is a two sided ideal of $\gr(A)$ and there is an isomorphism $\gr(A/I)\cong \gr(A)/\gr(I)$; see \cite{Bou2}, Chapter~3, \S~2.4. If $B$ is a subring of $A$, then $B$ inherits a filtration from $A$ and we have a canonical embedding $\gr(B)\hookrightarrow\gr(A)$.

Now assume that $\bigcup_iA_i=A$ and $A_i=\{0\}$ for $i<0$. For a nonzero $x\in A$ we define $\deg(x):=\min\{i\in\mb{N}\,|\,\,x\in A_i\}$ and $\gr(x):=x+A_{k-1}\in\gr(A)^k=A_k/A_{k-1}$ where $k=\deg(x)$. If $\gr(A)$ has no zero divisors, then the same holds for $A$ and we have for $x,y\in A\sm\{0\}$ that $\deg(xy)=\deg(x)+\deg(y)$ and $\gr(xy)=\gr(x)\gr(y)$. Now assume that $A$ is commutative and $\gr(A)$ is a domain, then $\gr((x))=(\gr(x))$ for all $x\in A\sm\{0\}$, where $(x)$ denotes the ideal of $A$ generated by $x$. Recall that $x$ of $A$ is called {\it prime} if $(x)$ is a prime ideal of $A$. It follows that $x\in A$ is prime if $\gr(x)\in\gr(A)$ is prime.

The universal enveloping algebra $U$ has a canonical $G$-stable filtration $k=U_0\subseteq U_1\subseteq U_2 \cdots$ and we have a canonical $G$-equivariant homomorphism $\gr(U)\to S(\g)$ of graded algebras which is an isomorphism by the PBW-theorem.
\begin{propnn}[{cf. \cite[1.2,1.3]{FrPa1}}]
There exists a $G$-equivariant, filtration preserving isomorphism
$U\stackrel{\sim}{\to}S(\g)$ of coalgebras such that the induced isomorphism $\gr(U)\stackrel{\sim}{\to}S(\g)$ of graded $G$-modules is the canonical one.
\end{propnn}
\begin{proof}
By \cite[Thm~1.2]{FrPa1} it is enough to show that $\g$ has a $G$-stable direct complement in $U$. Let $\tilde{G}$ be as in \cite[Sect.~4]{Prem} (or \cite[Sect.~6]{GoPr}). So $\tilde{G}$ is the direct product of the simple factors of $DG$ with the factors isomorphic to $\SL_m$, for some $m$ with $p|m$, replaced by $\GL_m$. Then $D\tilde{G}=DG$, since, by (H1), $DG$ is the direct product of its simple factors. By \cite[Lem.~6.2]{GoPr} and \cite[Lem.~4.1 proof]{Prem}\footnote{There is a flaw in the proof of \cite[Lem.~4.1]{Prem}, also noticed by Premet. For the $G_i$ of type $B,C,D$ one has to use the representation given by the vector representation of the corresponding classical group rather than the adjoint representation.} there exist a torus $\tilde{T}$, a toral subalgebra $\t_0$ of $\tilde{\t}=\Lie(\tilde{T})$ and a $DG$-equivariant embedding
$$\g\hookrightarrow\tilde{\g}\oplus\tilde{\t}=\Lie(\tilde{G}\times\tilde{T})\eqno{(*)}$$
of restricted lie algebras, such that $\tilde{\g}\oplus\tilde{\t}=\g\oplus\t_0$.

We now follow \cite[3.3]{Prem}; see also \cite[6.3]{PrSkr}. Modifying the definition there slightly, we say that a linear algebraic group $G'$ has {\it Richardson's property} if there exists a finite dimensional representation $G'\to\GL(V)$ such that the corresponding representation $\g'=\Lie(G')\to\gl(V)$ is faithful and $\g'$ has a $G'$-stable complement in $\gl(V)$. Using that $p$ is good one easily checks that $\tilde{G}$ has Richardson's property; the same holds for $\tilde{G}\times\tilde{T}$, see e.g. \cite[I.3]{Sl2}. Combining this with (*) we obtain a $DG$-equivariant embedding $\g\oplus\t_0\hookrightarrow\gl(V)$ of restricted Lie algebras for some finite dimensional $DG$-module $V$, such that $\g\oplus\t_0$ has a $DG$-stable complement in $\gl(V)$. Adding $\t_0$ to this complement we obtain a $DG$-stable complement for $\g$ in $\gl(V)$. By the universal property of $U$ we can extend the embedding $\g\hookrightarrow\gl(V)$ to a homomorphism of associative algebras $U\to\End_k(V)$. This homomorphism is $DG$-equivariant and if we combine it with the projection $\gl(V)\to\g$, then we obtain a projection $U\to \g$ of $DG$-modules. Since the centre of $G$ acts trivially on $U$, this is a projection of $G$-modules.
\end{proof}

The preceding proposition implies that each $G$-module $U_{n-1}$ has a $G$-stable direct complement in $U_n$. So we obtain the following
\begin{cornn}
The canonical embeddings $\gr(U^G)\hookrightarrow S(\g)$ and $\gr(Z)\hookrightarrow S(\g)$ induce isomorphisms of algebras
\begin{align*}
\gr(U^G)&\stackrel{\sim}{\to}S(\g)^G\text{\quad and}\\
\gr(Z)&\stackrel{\sim}{\to}S(\g)^\g.
\end{align*}
\end{cornn}

\subsection{A restriction theorem for U}\label{ss.Urestriction}
Take $\rho\in X(T)$ such that $\rho|_{(T\cap DG)}$ is the half sum of the positive roots in $X(T\cap DG)$; note that this makes sense because of assumption (H1). We have $d\rho\in\t^*(\mb F_p)$ and we will simply write $\rho$ instead of $d\rho$. Then $\rho(h_\alpha)=1$ for all simple roots $\alpha$. Define the {\it dot action} of $W$ on $\t^*$ by $$w\dt\lambda=w(\lambda+\rho)-\rho.$$
The corresponding action of $W$ on $S(\t)$ is given by $s_\alpha\dt h=s_\alpha(h)-\alpha(h)$ for $h\in\t$ and $\alpha$ simple. This shows that the dot action on $S(\t)$ and $\t^*$ is independent of the choice of $\rho$.
Let $\gamma$ be the comorphism of the isomorphism $\lambda\mapsto\lambda-\rho:\t^*\stackrel{\sim}{\to}\t^*$ of varieties. We have $\gamma(h)=h-\rho(h)$ for all $h\in\t$, $w(\lambda)-\rho=w\dt(\lambda-\rho)$ and $\gamma(w\dt x)=w(\gamma(x))$ for all $\lambda\in\t^*$, all $x\in S(\t)$ and all $w\in W$. So $\gamma$ induces an isomorphism
$$\gamma:S(\t)^{W\nts\dt}\stackrel{\sim}{\to}S(\t)^W.$$

Let $\Psi\,\,\colon U=U(\u)\ot U(\t)\ot U(\u^+)\longrightarrow\, U(\t)=S(\t)$ be the linear map taking $x\ot h\ot y$ to $x^0hy^0$, where $u^0$ denotes the scalar part of $u\in U$ with respect to the decomposition $U=K1\oplus\, U_+$ where $U_+$ is the augmentation ideal of $U$. The restriction of $\Psi$ to $U^{N_G(T)}$ is an algebra homomorphism. Using the descriptions of $\Phi$ and $\Psi$ and a PBW-basis it follows that for  $x\in U\sm\{0\}$ with $\Phi(\gr(x))\neq 0$ we have $\Psi(x)\neq 0$ and
\begin{equation}\label{eqn.gr}
\gr(\gamma(\Psi(x)))=\gr(\Psi(x))=\Phi(\gr(x)).
\end{equation}
As we have seen in Subsection~\ref{ss.regsselts}, the restriction of $\Phi$ to $S(\g)^G$ is injective. It follows that the restriction of $\Phi$ to $U^G$ is injective and that the displayed equalities hold for all $x\in U^G$. We can also deduce that $\Psi(U^G)=S(\t)^{W\nts\dt}$ from the fact that $\Phi(S(\g)^G)=S(\t)^W$; see the proof of Proposition~2.1 in \cite{Ve}. So we obtain an isomorphism
$$\Psi:U(\g)^G\stackrel{\sim}{\to}S(\t)^{W\nts\dt}.$$
This was also proved in \cite[Lem.~5.4]{KW} under the only condition that $h_\alpha\ne0$ for all $\alpha\in\Sigma$.
%this is implied by: $DG$ does not have an adjoint group of type $B$ as a simple factor in case $p=2$.
We have $\rho(h^{[p]})=\rho(h)^p$ and $\gamma(\eta(h))=\eta(h)$ for all $h\in\t$. The homomorphisms $\eta$, $\Phi$ and $\Psi$ are related by
\begin{equation}\label{eqn.eta}
\eta\circ\Phi=\Psi\circ\eta :S(\g)\to S(\t).
\end{equation}

\subsection{Veldkamp's theorem}\label{ss.veldkamp}
Recall that $d\alpha\ne0$ for all $\alpha\in\Sigma$. We have seen in Subsection~\ref{ss.regsselts} that $p$ is not a torsion prime for $G$. Therefore $S(\t^*)^W$ is a graded polynomial algebra by \cite[Cor. to Thm.~2, Thm.~3]{Dem}. By (H3) we have an isomorphism $S(\t)^W\cong S(\t^*)^W$ of graded algebras, so $S(\t)^W$ is a graded polynomial algebra. Let $\sigma_1,\ldots,\sigma_r$ be algebraically independent homogeneous generators of $S(\t)^W$. Put $s_i=\Phi^{-1}(\sigma_i)\in S(\g)^G$ and put $u_i=\Psi^{-1}(\gamma^{-1}(\sigma_i))\in U^G$. Note that $\gr(u_i)=s_i$ by \eqref{eqn.gr}. The following theorem was first proved in \cite{Ve} under much stronger conditions on $p$.
\begin{thmnn}[{cf. \cite{Ve}, \cite{KW}, \cite{MiRu}, \cite{BrGo}}]\
\begin{enumerate}[{\rm(i)}]
\item $S(\g)^\g$ is a free $S(\g)^p$-module with basis $\{s_1^{k_1}\cdots s_r^{k_r}\,|\, \, 0\leq k_i<p\}$.
\item $S(\g)^\g\cong S(\g)^p\otimes_{(S(\g)^p)^G}S(\g)^G$.
\item $Z$ is a free $Z_p$-module with basis $\{u_1^{k_1}\cdots u_r^{k_r}\,|\, \, 0\leq k_i<p\}$.
\item $Z\cong Z_p\otimes_{Z_p^G}U^G$.
\end{enumerate}
\end{thmnn}
\noindent Assertions (ii) and (iv) are immediate consequences of (i) and (iii) respectively, since the bases consist of $G$-invariants. Assertion (i) can be proved using the differential criterion for regularity (see \cite[7.14]{Jan3}, \cite[3.14]{Sl1})
%In case char 2 is a problem (cf \cite{BrGo}) one could use the arguments in the proof of the proposition in Sect 1.4 above to cover that case.
and a result of Skryabin \cite[Thm.~5.4]{Skry}. See \cite[Sect.~3]{PrT} for more details. Assertion (iii) can be deduced from (i) as in \cite{Ve}.

In Section~\ref{s.rationality} we will need a geometric version of assertion (iv) of the above theorem. To state it we need some notation. Let $\xi:\t^*\to\t^{(1)*}$ be the morphism induced by the homomorphism $\eta:S(\t^{(1)})\to S(\t)$ from Subsection~\ref{ss.pcentre} and let $\zeta^{(1)*}:\g^{(1)*}\to\t^{(1)*}/W$ be the morphism induced by the homomorphism $k[\t^{(1)*}]^W\stackrel{\sim}{\to}k[\g^{(1)*}]^G\hookrightarrow k[\g^{(1)*}]$. We have $\xi(\lambda)(h)=\lambda(h)^p-\lambda(h^{[p]})$ for all $\lambda\in\t^*$ and $h\in\t$. For $\lambda\in\t^*(\mb F_p)$ we have $\lambda(h^{[p]})=\lambda(h)^p$ for all $h\in\t$. Therefore $\xi(\lambda)=0$ for all $\lambda\in\t^*(\mb F_p)$ and $\xi(w\dt\lambda)=\xi(w(\lambda))=w(\xi(\lambda))$ for all $\lambda\in\t^*$, $h\in\t$ and $w\in W$. Furthermore, we have $\zeta^{(1)*}(\chi)=\pi(\chi_s')$, where $\chi_s'$ is a conjugate of the semisimple part $\chi_s$ of $\chi$ that lies in $\t^{(1)*}$ and $\pi:\t^{(1)*}\to\t^{(1)*}/W$ is the canonical morphism.

\begin{cornn}[{cf. \cite[Sect.~3]{Kry}, \cite[Cor.~3]{MiRu}}]
Let $\xi$ and $\zeta^{(1)*}$ be as defined above. There is a canonical $G$-equivariant isomorphism
$$\mc Z\stackrel{\sim}{\to}\g^{*(1)}\times_{\t^{(1)*}/W}\t^*/W{\nts\dt}\ .$$
Here the $G$-action on the fibre product comes from the adjoint action on the first factor, the morphism $\t^*/W{\nts\dt}\to\t^{(1)*}/W$ is induced by $\xi$ and the morphism $\g^{*(1)}\to\t^{(1)*}/W$ is the composite of $\chi\mapsto\chi^p:\g^{*(1)}\stackrel{\sim}{\to}\g^{(1)*}$ and $\zeta^{(1)*}$.
\end{cornn}
Identify $\mc Z$ with a closed subvariety of $\g^{*(1)}\times\t^*/W{\nts\dt}$ by means of the above isomorphism. Let $\pi_\dt:\t^*\to\t^*/{W\nts\dt}$ be the canonical morphism. If $\chi\in\g^{*(1)}$, $\lambda\in\t^*$ and $\chi_s'$ is a conjugate of the semisimple part $\chi_s$ of $\chi$ that lies in $\t^*$, then $(\chi,\pi_\dt(\lambda))\in\mc Z$ if and only if $\xi(\lambda)=w(\chi_s'^p)$ for some $w\in W$, that is, if and only if for some $w\in W$ we have
$$\lambda(h)^p-\lambda(h^{[p]})=w(\chi_s')(h)^p\text{\quad for all\ }h\in\t\ .$$
\begin{remsnn}
1.\ In the above description of the centre as a fibre product we could have avoided working with $\g^{*(1)}$ and the isomorphism $\g^{*(1)}\stackrel{\sim}{\to}\g^{(1)*}$, but the advantage is that now the first component $\chi$ of a point $(\chi,\pi_\dt(\lambda))$ of $\mc Z$ can be interpreted as the $p$-character. More precisely, it is the $p$-character of any irreducible $\g$-module whose central annihilator corresponds to $(\chi,\pi_\dt(\lambda))$.\\
2.\ Let $\chi\in\g^*$ and $\lambda\in\t^*$ such that $\chi\in\b^*$ (i.e. $\chi(\u^+)=0$) and $\lambda(h)^p-\lambda(h^{[p]})=\chi(h)^p$ for all $h\in\t$. Then $(\chi,\pi_\dt(\lambda))\in\mc Z$ and the corresponding maximal ideal is the central annihilator of the baby Verma module $Z_\chi(\lambda)$, see \cite[6.7,6.8]{Jan2}. For any point $(\chi,\pi_\dt(\lambda))\in\mc Z$ there exists a $g\in G$ such that $g\cdot\chi$ and $\lambda$ satisfy the above two conditions.\\
\end{remsnn}

\section{Rationality}\label{s.rationality}
In this section we take a geometric viewpoint and we will identify $\mc Z$ with the closed subvariety of $\g^{*(1)}\times\t^*/W{\nts\dt}$ described in Subsection~\ref{ss.veldkamp}.
To prove Theorem~\ref{thm.rat} below we need to ``untwist the $G$-action on $\mc Z$". Since $G$, $\g$ and the adjoint action of $G$ on $\g$ are defined over $\mb F_p$, there exists a $k$-algebra isomorphism $S(\g)^{(1)}\stackrel{\sim}{\to}S(\g)$ such that $g\cdot x$ is mapped to $Fr(g)\cdot x$; see e.g. \cite[I.9.10]{Jan1}. We denote the action of $G$ on $k[\g^{*(1)}]\cong k[\g^{(1)*}]=S(\g)^{(1)}$ that corresponds via this isomorphism to the adjoint action of $G$ on $S(\g)$ by $g\star x$ and call it the {\it star action} of $G$ on $\g^{*(1)}$ (this is the corresponding variety). We have $$Fr(g)\star\chi=g\cdot\chi\ \text{\quad for all\ }g\in G\text{\ and\ }\chi\in\g^{*(1)}.$$
From the corollary to Veldkamp's theorem we deduce that the star action extends to an action on the Zassenhaus variety by $g\star(\chi,\pi_\dt(\lambda))=(g\star\chi,\pi_\dt(\lambda))$. One can also deduce the existence of the star action on ${\mc Z}$ directly from assertion (iv) of Veldkamp's theorem and the isomorphism $\eta:S(\g)^{(1)}\stackrel{\sim}{\to}Z_p$. Of course the star and ordinary action of $G$ on $\mc Z$ are related by the same equation as above. Put $\mc Z_{\rm rss}=pr_1^{-1}(\g^{*(1)}_{\rm rss})$, where $pr_1:\mc Z\to\g^{*(1)}$ is the projection on the first component (the ``$p$-character").

\begin{thmgl}\label{thm.rat}
Put $F=\Phi^{-1}\big(\prod_{\alpha\in\Sigma}(h_\alpha^p-h_\alpha)\big)\in k[\g^*]^G$. Then the variety $\mc Z_{\rm rss}$ with the star action is $G$-equivariantly isomorphic to $\{\chi\in\g^*\,|\,F(\chi)\ne0\}$. In particular, ${\rm Frac}(Z)$ is rational.
\end{thmgl}

\begin{proof}
Recall that an element $x$ of $\g$ is toral if $x^{[p]}=x$. Let $(h_1,\ldots,h_r)$ be an $\mb F_p$-basis of $\t(\mb F_p)$. Then it is also a $k$-basis of $\t$ and the $h_i$ are toral. We write $\chi_i$ for the functional $\chi\mapsto \chi(h_i)^p$ and we write $\lambda_i$ for the functional $\lambda\mapsto\lambda(h_i)$. Then $k[\t^{*(1)}]$ is a polynomial algebra in the $\chi_i$ and $k[\t]$ is a polynomial algebra in the $\lambda_i$. Let $C$ be the closed subvariety of $G/T\times\t^{*(1)}_{\rm reg}\times\t^*$ defined by the equations
\begin{align}\label{eqn.C1}
\chi_i=\lambda_i^p-\lambda_i\text{\quad for\ }i=1,\ldots,r.
\end{align}
Let the morphism $\xi:\t^*\to\t^{(1)*}$ be defined as in Subsection~\ref{ss.veldkamp}. Then $(gT,\chi,\lambda)\in C$ if and only if
\begin{align}\label{eqn.C2}
\xi(\lambda)=\chi^p\ .
\end{align}
We define an action of $W$ on $C$ as follows $w\cdot(gT,\chi,\lambda)=(gw^{-1}T,w(\chi),w\dt\lambda)$. This action commutes with the $G$-action on $C$ that comes from the action by left multiplication on the first factor. Put $H=\Phi(F)=\prod_{\alpha\in\Sigma}(h_\alpha^p-h_\alpha)\in k[\t^*]^W$ and put
$$\widetilde{O}=\{(gT,\chi)\in G/T\times\t^*\,|\,H(\chi)\ne0\}.$$
Then $\widetilde{O}$ is an open subset of $G/T\times\t^*_{\rm reg}$ which is stable under $G$ and $W$, the actions being as described in Subsection~\ref{ss.regsselts}. Now $(gT,\chi,\lambda)\mapsto(gT,\lambda+\rho)$ is an isomorphism from $C$ to $\widetilde{O}$ which is equivariant for the actions of $G$ and $W$. Its inverse is given by
$$(gT,\lambda)\mapsto(gT,\xi(\lambda)^{1/p},\lambda-\rho).$$
Note that this is indeed a morphism, since the map $\lambda\mapsto\xi(\lambda)^{1/p}:\t^*\to\t^{*(1)}$ is obtained by composing the morphism $\xi:\t^*\to\t^{(1)*}$ with the inverse of the isomorphism $\lambda\mapsto\lambda^p:\t^{*(1)}\to\t^{(1)*}$ (which is $\lambda\mapsto\lambda^{1/p}$). Of course it is also clear from \eqref{eqn.C1} that dropping the $\chi$-component and shifting the $\lambda$-component by $\rho$ must be an isomorphism.
It follows that $C/W$ is G-equivariantly isomorphic to $\widetilde{O}/W$. Now put $$O=\{\chi\in\g^*\,|\,F(\chi)\ne0\}.$$ Then the $W$-quotient morphism $G/T\times\t^*_{\rm reg}\to\g^*_{\rm reg}$ induces a $W$-quotient morphism $\widetilde{O}\to O$ and we obtain a $G$-equivariant isomorphism $\widetilde{O}/W\stackrel{\sim}{\to}O$.

From the description of $\mc Z$ in Subsection~\ref{ss.veldkamp} it is clear that we have a surjective morphism $C\to\mc Z_{\rm rss}$ given by $(gT,\chi,\lambda)\mapsto (g\star\chi,\pi_\dt(\lambda))$. This morphism is constant on $W$-orbits and it is $G$-equivariant if we give $\mc Z$ the star action. To prove the theorem it suffices to show that this morphism is separable and that its fibers are single $W$-orbits, because then we have by \cite[Prop.~II.6.6]{Bo} that $\mc Z_{\rm rss}\cong C/W\cong O$, $G$-equivariantly. That the fibers are single $W$-orbits follows easily from \eqref{eqn.C2}, the $W$-equivariance property of $\xi$ (see Subsection~\ref{ss.veldkamp}) and the fact that the elements of $\t^{*(1)}_{\rm reg}$ have trivial stabiliser in $W$. It remains to show that the morphism $C\to{\mc Z}_{\rm rss}$ is separable. We have the following commutative diagram, where the maps are described on the right.

\begin{equation*}
\xymatrix{
C\ar[r]\ar@{->>}[d]&G/T\times\t^{*(1)}_{\rm reg}\ar@{->>}[d]\\
{\mc Z}_{\rm rss}\ar[r]&\g^{*(1)}_{\rm rss}
}
\quad
\quad
\xymatrix{
(gT,\chi,\lambda)\ar@{|->}[r]\ar@{|->}[d]&(gT,\chi)\ar@{|->}[d]\\
(g\star\chi,\pi_\dt(\lambda))\ar@{|->}[r]&g\star\chi
}
\end{equation*}
Since the polynomial $X^p-X-a\in L[X]$ is separable for any extension field $L$ of $k$ and any $a\in L$ we have, by \eqref{eqn.C1}, that the morphism $C\to G/T\times\t^{*(1)}_{\rm reg}$ is separable. Since $\g^{*(1)}_{\rm reg}$ with star action is isomorphic to $\g^*_{\rm reg}$ we have, by the results in Subsection~\ref{ss.regsselts}, that the morphism $G/T\times\t^{*(1)}_{\rm reg}\to\g^{*(1)}_{\rm rss}$ is separable. It follows that the morphisms $C\to\g^{*(1)}_{\rm rss}$ and $C\to{\mc Z}_{\rm rss}$ are separable.
\end{proof}

\begin{remsnn}
1.\ The proof shows that the morphism ${\mc Z}_{\rm rss}\to\g^{*(1)}_{\rm rss}$ is separable. This means that the field extension ${\rm Frac}(Z_p)\subseteq{\rm Frac}(Z)$ is separable. This was pointed out in \cite[Lemma~4.2]{KW} and it also follows from \cite[Prop.~3.14]{BrGo} and an elementary result in algebraic geometry (see e.g. \cite[Thm.~5.1.6(iii)]{Spr}).\\
2.\ The constructed isomorphism is also a $G$-equivariant isomorphism between $\mc Z_{\rm rss}$ with ordinary $G$-action and $\{\chi\in\g^{*[1]}\,|\,F(\chi)\ne0\}$, see Subsection~\ref{ss.pcentre}. Of course the latter variety is $G$-equivariantly isomorphic with $\{\chi\in\g^{*(1)}\,|\,F^p(\chi)\ne0\}$, but then one has to compose with the ``less natural" $G$-equivariant isomorphism $\g^{*[1]}\cong\g^{*(1)}$.\\
3.\ Assume $G=\SL_2$ and let $(h,e,f)$ be the standard basis of $\spl_2$ (so $[h,e]=2e$, $[h,f]=-2f$, $[e,f]=h$). If $p=2$, then $Z$ is a polynomial algebra generated by $h$, $e^2$ and $f^2$. Now assume that $p>2$. Then $Z$ is generated by $x=e^p, y=f^p, z=h^p-h\in Z_p$ and $c=4fe+(h+1)^2\in U^G$ subject to the relation $c^p-2c^{(p+1)/2}+c=4xy+z^2$. This can be seen using the restriction theorem for $U^G$. Note that $4xy+z^2$ is equal to $\eta$ applied to the element $4ef+h^2\in S(\g)^G$. Now $c^p-2c^{(p+1)/2}+c=c(c^{(p-1)/2}-1)^2$. So if we put $u=x/(c^{(p-1)/2}-1)$, $v=y/(c^{(p-1)/2}-1)$ and $w=z/(c^{(p-1)/2}-1)$, then $c=4uv+w^2$. Therefore $Z[(c^{(p-1)/2}-1)^{-1}]\cong k[u,v,w][((4uv+w^2)^{(p-1)/2}-1)^{-1}]$. The localisation of $Z$ in Theorem~\ref{thm.rat} is slightly bigger. There we make $c^p-2c^{(p+1)/2}+c=4xy+z^2$ invertible (cf. the proof of Theorem~\ref{thm.ufd} below). I am grateful to A.\ Premet for pointing this out to me.\\
4.\ The isomorphism $O\stackrel{\sim}{\to}\mc Z_{\rm rss}$ from Theorem~\ref{thm.rat} can be extended to a morphism $\g^*\to\mc Z_{\rm rss}$. Define a $p$-mapping on $\g^*$ to be a morphism from $\g^*$ to $\g^*$ such that for every maximal torus $T_1$ of $G$ it leaves $\t_1^*\subseteq\g^*$ stable and restricts on it to a $p$-semilinear map which is the identity on the $\mb F_p$-defined points for the unique split $\mb F_p$-structure on $T_1$. By the density of the semisimple elements such a map is unique. We now show that it exists. By hypothesis (H3) we have an isomorphism $\g\cong\g^*$ of $G$-modules. By \cite[I.7.16]{Jan1} and \cite[Chap.~V Annexe, Prop.~1]{Bou1} we have that there exists an $\mb F_p$-defined isomorphism $\g\cong\g^*$ of $G$-modules. Now we carry the $p$-mapping of $\g$ to $\g^*$ using such an isomorphism. One easily checks that this is a $p$-mapping on $\g^*$ and that it is $G$-equivariant. We denote it by $\chi\mapsto\chi^{[p]}:\g^*\to\g^*$.

Now let $\varphi:\g^*\stackrel{\sim}{\to}\g^{*(1)}$ be the isomorphism given by the $\mb F_p$-structure of $G$ that we fixed at the beginning of Section~\ref{s.prelim}. This isomorphism intertwines the $G$ action on $\g^*$ with the star action of $G$ on $\g^{*(1)}$. Recall that the isomorphism $\lambda\mapsto\lambda-\rho:\t^*\stackrel{\sim}{\to}\t^*$ intertwines the ordinary action of $W$ with the dot action of $W$. Therefore it induces an isomorphism $\t^*/W\stackrel{\sim}{\to}\t^*/W{\nts\dt}$ which we will also denote by $-\rho$. Finally, let $\zeta^*:\g^*\to\t^*/W$ be the canonical morphism (cf. the definition of $\zeta^{(1)*}$ in Section~\ref{ss.veldkamp}). Then the proof of Theorem~\ref{thm.rat} shows that the morphism
$$\chi\mapsto(\varphi(\chi^{[p]}-\chi),\zeta^*(\chi)-\rho):\g^*\to\mc Z$$
restricts to the isomorphism $O\stackrel{\sim}{\to}\mc Z_{\rm rss}$ from the Theorem. I am grateful to R.~Bezrukavnikov for pointing out that such a formula is suggested by the proof of Theorem~\ref{thm.rat}.

%1)The p-mapping is also G-equivariant for the star action. We deduce that the p-mapping and \varphi (considered as a map:\g^*\to\g^*) commute.
%2)Note that on any F_p-defined and split torus \varphi and the p-mapping are each others inverse.
%3)From Krylyuks paper we easily deduce that the displayed morphism does not map the regular locus of \g^* into the smooth locus of \mc Z.
%4)Let \psi:\g\to\g^(1) be the iso given by the F_p-structure, then we have \varphi(f)(\psi(x))^p=f(x) for all x\in\g and f\in\g^*.
\end{remsnn}

\section{Unique Factorisation}\label{s.ufd}
In this section we take an algebraic viewpoint and we will work directly with the definition $\mc Z=\Maxspec(Z)$, so that $k[\mc Z]=Z$.
The following lemma is well-known. Recall that an element of a module for a group is called a {\it semiinvariant} if the subspace that it spans is stable under the group action.
\begin{lemgl}\label{lem.invUFD}
Let $G'$ be a connected linear algebraic group acting by automorphisms on a commutative $k$-algebra $A$ such that the representation of $G'$ on $A$ is rational. Assume that $A$ is a unique factorisation domain and that every semiinvariant of $G'$ in $A$ is an invariant. Then $A^{G'}$ is a unique factorisation domain and its irreducible elements are the $G'$-invariant irreducible elements of $A$.
\end{lemgl}

\begin{cornn}
$S(\g)^G$ is a unique factorisation domain and its irreducible elements are the $G$-invariant irreducible elements of $S(\g)$.
\end{cornn}
\begin{proof}
By the same arguments as in \cite[Lem.~2]{T2} and the proof of \cite[Prop.~3(2)]{T2} we get that every $G$-semiinvariant in $S(\g)$ is an invariant. The idea of these arguments comes, of course, from the treatment of semisimple elements in $\g^*$ in \cite[Sect.~3]{KW}.
\end{proof}

\begin{thmgl}\label{thm.ufd}
The centre $Z$ is a unique factorisation domain.
\end{thmgl}
\begin{proof}
Put $H=\prod_{\alpha\in\Sigma}(h_\alpha^p-h_\alpha)$, $F=\Phi^{-1}(H)$ and for $a\in{\mb F}_p$ put $H_a=\prod_{\alpha\in\Sigma}(h_\alpha-a)$. Then $H=\prod_{a\in{\mb F}_p}H_a$. Put $F_0=\Phi^{-1}(H_0)$. Then  $\chi\in\g^*$ is regular semisimple if and only if $F_0(\chi)\ne0$. So $\g^{*(1)}_{\rm rss}$ is the open subset of $\g^{*(1)}$ defined by $F_0^p$ ($p$-th power as a function on $\g^*$). Now the isomorphism from Subsection~\ref{ss.veldkamp} has
$$x\ot y\mapsto\eta(x^{1/p})\ot\Psi^{-1}(y):
k[\g^{*(1)}]\ot_{k[\t^{(1)*}]^W}k[\t^*]^{W\dt}\stackrel{\sim}{\to}Z_p\otimes_{Z_p^G}U^G$$
as a comorphism, so $\mc Z_{\rm rss}$ is the open subset of $\mc Z$ defined by $\eta(F_0)$. Since, by Theorem~\ref{thm.rat}, $Z[\eta(F_0)^{-1}]\cong k[\mc Z_{\rm rss}]\cong k[\g^*][F^{-1}]$ is a UFD, it suffices by Nagata's Lemma (see e.g. \cite[Lem~19.20]{Eis} or \cite[Thm.~3.7(i)]{Cohn}) to show that $\eta(F_0)$ is a product of prime elements in $Z$.

Let $\Sigma/W$ denote the set of $W$-orbits in $\Sigma$. For $\Gamma\in\Sigma/W$ and $a\in{\mb F}_p$ put $$H_a^\Gamma=
\begin{cases}
\prod_{\alpha\in\Gamma}(h_\alpha-a)\text{\quad if\ }p\ne2\\
\prod_{\alpha\in\Gamma,\alpha>0}(h_\alpha-a)\text{\quad if\ }p=2.
\end{cases}
$$
Then
$$H=
\begin{cases}
\prod_{a\in{\mb F}_p,\Gamma\in\Sigma/W}H_a^\Gamma\text{\quad if\ }p\ne2\\
\prod_{a\in{\mb F}_p,\Gamma\in\Sigma/W}(H_a^\Gamma)^2\text{\quad if\ }p=2.
\end{cases}
$$
Using \eqref{eqn.eta} we get
$$\eta(F_0)=\eta(\Phi^{-1}(H_0))=\Psi^{-1}(\eta(H_0))=\Psi^{-1}(H)=\Psi^{-1}(\gamma^{-1}(H)).$$
So it is enough to show that the elements $\Psi^{-1}(\gamma^{-1}(H_a^\Gamma))$ are prime in $Z$.
Now recall from Subsection~\ref{ss.filtration} that we have a canonical isomorphism $\gr(Z)\cong S(\g)^\g$. So it is enough to show that the elements $\gr(\Psi^{-1}(\gamma^{-1}(H_a^\Gamma)))$ are prime in $S(\g)^\g$. By \eqref{eqn.gr} we have
$$\gr(H_a^\Gamma)=\Phi(\gr(\Psi^{-1}(\gamma^{-1}(H_a^\Gamma)))).$$
So $$\gr(\Psi^{-1}(\gamma^{-1}(H_a^\Gamma)))=\Phi^{-1}(\gr(H_a^\Gamma))=\Phi^{-1}(H_0^\Gamma).$$
From the simply connectedness of $DG$ we deduce that $h_\alpha$ and $h_\beta$ are linearly independent for all $\alpha,\beta\in\Sigma$ with $\alpha\ne\pm\beta$. From this one easily deduces that the elements $H_0^\Gamma$ are irreducible in $S(\t)^W$.
\begin{comment}
\hfil\break
BEGIN COMMENT\\
First note that for $\alpha,\beta>0$ with $\alpha\ne\beta$ the elements $h_\alpha$ and $h_\beta$ are mutually coprime in the UFD $S(\t)$. If an element of the form $x=c\prod_{\alpha>0}h_\alpha^{n_\alpha}$, $c\in k^\times$ and $n_\alpha\in\mathbb{N}$, is $W$-invariant then we must have $n_\alpha=n_\beta$ whenever $\alpha,\beta>0$ are $W$-conjugate (note that the set of elements of this form is $W$-stable). Now let $\Gamma\in\Sigma/W$. Then $y_\Gamma:=\prod_{\alpha\in\Gamma,\alpha>0}h_\alpha$ is anti-invariant (one only has to check that for $\alpha$ simple $s_\alpha(y_\Gamma)=-y_\Gamma$). So if $p=2$, then we get that $x$ is a scalar multiple of a product of $y_\Gamma$'s. If $p\ne2$, then we deduce that the $n_\alpha$ are even (one only has to check this for $\alpha$ simple). So then we get that $x$ is a scalar multiple of a product of elements of $y_\Gamma^2$'s.\\
END COMMENT\\
\end{comment}
So the elements $\Phi^{-1}(H_0^\Gamma)$ are irreducible in $S(\g)^G$. But then they are irreducible in $S(\g)$ by the corollary to Lemma~\ref{lem.invUFD}. So they are also irreducible in $S(\g)^\g$ which is a UFD by \cite[Prop.~3(2)]{T2}.
\end{proof}

Combining Theorem~\ref{thm.ufd} with Braun's work \cite{Braun2} we obtain assertion (ii) of the following corollary. For a definition of Calabi-Yau algebra, see \cite[Sect.~1]{Braun2}. The definition there is based on the definitions and results in \cite[Sect.~3]{IR}, but the symmetry property is required globally rather than locally (see also J.~Miyachi's example in \cite{IR} after Theorem~3.3).

\begin{cornn}\
\begin{enumerate}[{\rm(i)}]
\item Every height one prime of $U$ is generated by a single central element.
\item $U$ is a Calabi-Yau algebra over any polynomial subring of $Z$ over which $Z$ (and therefore $U$) is module finite.
%\item Let $n$ be the dimension of $\g$. For all finite dimensional $\g$-modules $V$ and $W$ we have
%$$\Ext^i_U(V,W)\cong \Ext^{n-i}_U(W,V)^*\quad 1\le i\le n.$$
\end{enumerate}
\end{cornn}
\begin{proof}
(i).\ The arguments are standard (see \cite[Rem.~(4) p136]{BrHa}); we give them for convenience of the reader. Let $Q$ be a height one prime of $U$. Put $q=Q\cap U$. Then $q$ is a height one prime of $Z$, see e.g. \cite[Thm.~13.8.14]{McCR}. Since $Z$ is a UFD $q$ is principal, so all we need to show is that $Q$ is generated by $q$. Since $U_q$ (the elements of $Z\sm q$ made invertible) is Azumaya by \cite[proof of Prop.~4.8]{BG}, we have that $Q_q=U_qq_q$, see e.g. \cite[Prop.~13.7.9]{McCR}. But then $Q=Uq$ by \cite[Lemma~6]{Braun1}. As pointed out in \cite[p136]{BrHa} one can also show this under the weaker assumption that $Z$ is locally factorial. The point is that one can apply \cite[Lemma~6]{Braun1} also to localisations $U_m$ at maximal ideals $m$ of $Z$ containing $q$ (the case $q\nsubseteq m$ is trivial).\\
(ii).\ This follows from (i) and \cite[Thm~2.16]{Braun2}.\\
%(iii). This follows from \cite[Prop~2.29]{Braun2}.
\end{proof}

\begin{remsnn}
1.\ The isomorphism $\Ext^i_U(V,W)\cong \Ext^{n-i}_U(W,V)^*$, $n=\dim(\g)$, $V$ and $W$ finite dimensional $U$-modules, from \cite[Prop~2.29]{Braun2} is in our case an easy consequence of Poincare duality for Lie algebra cohomology.\\
\begin{comment}
COMMENT\\
One needs that $\ad x\in\g$ has trace $0$ for all $x\in\g$.
In general we have
\begin{enumerate}[(1)]
\item $H_n(\g,V^*)=\Tor_U^n(V^*,k)\cong\Ext_U^n(k,V)^*=H^n(\g,V)^*$, by the def of $H_n(\g,-)$ and $H^n(\g,-)$ \cite[p.270]{CE} and a relation between $\Ext$ and $\Tor$ \cite[Prop.~3.3, p211]{CE}.
\item $H_{n-k}(\g,M)\cong H^k(\g,\bigwedge^{\nts n}\g\otimes_kM)$, $0\le k\le n$. See \cite[Ex.~15 p288]{CE}.
\end{enumerate}
Finally, note that one can prove that $H_n(\g,M\ot_kN)\cong\Tor^U_n(N,M)$ and that $H^n(\g,\Hom_k(M,N)\cong\Ext^n_U(M,N)$ using \cite[Prop.~5.1 p277 and Cor.~4.4 p170]{CE}.\\
COMMENT\\
\end{comment}
2.\ From the proof of \cite[Prop.~1.2]{FrPa2} and some standard facts about Calabi-Yau algebras
%see e.g.\cite[Prop.~2.4]{BGS}(This proposition puzzles me a little. How could they get the global symmetry property from the local one?)
it follows immediately that $U$ is Calabi-Yau over $Z_p$. K.~A.~Brown informed me that it seems plausible that one can deduce from this by general arguments that $U$ is Calabi-Yau over any polynomial subring of $Z$ over which $Z$ is module finite.\\
3.\ The method used in \cite{PrT} essentially amounted to showing that the field extension $k(\g)^G\subseteq k(\g)$ is purely transcendental. It would imply that the invariant algebra $k[\g]^\g$ has a rational field of fractions. Recently Colliot-Thelene, Kunyavskii, Popov and Reichstein \cite{CKPR} showed that in characteristic zero the field extension $k(\g)^G\subseteq k(\g)$ can only be purely transcendental in types $A$, $C$, $G_2$ (and also that it is indeed the case in type $C$). By modifying their arguments Premet showed in \cite{Prem2} that also in characteristic $p>>0$ the field extension $k(\g)^G\subseteq k(\g)$ can only be purely transcendental in types $A$, $C$, $G_2$. Furthermore, he deduced from this and from Theorem~\ref{thm.rat} in this paper, using reduction mod $p$ arguments, that the Gelfand-Kirillov conjecture cannot hold for simple complex Lie algebras that are not of type $A$, $C$ or $G_2$.
%combine \cite{Prem2} Prop 4.1, Prop. 4.4 and the arguments after it.

Similarly, the invariant algebra $k[G]^\g$ would have a rational field of fractions if the field extension $k(G)^G\subseteq k(G)$ were purely transcendental. We note that the method used to prove the rationality of ${\rm Frac}(Z)$ in this paper does not apply in these cases, since the field extensions $k(\g)^p\subseteq k(\g)^\g$ and $k(G)^p\subseteq k(G)^\g$ are not separable (in fact they are purely inseparable).

The algebras $k[\g]^\g$ and $k[G]^\g$ are known to be unique factorisation domains, see \cite{T2} and \cite[Sect.~2]{T3} (The case $k[\g]^\g$ is very easy and essentially due to S.\ Skryabin).\\
4.\ It is not clear whether the results of this paper can be extended to ``simply connected" quantized enveloping algebras at a root of unity as defined by De Concini, Kac and Procesi. The main problem is that the quantum coadjoint action of $\tilde G$ on $\Omega=\Maxspec(Z_0)$ is not a morphic action of an algebraic group on the algebraic variety $\Omega$, but an action of an infinite dimensional Lie group on the complex analytic variety $\Omega$. It is known that the centres of these quantized enveloping algebras are always locally factorial, see \cite[Thm.~24]{BrHa}. For results on rationality and factoriality for the centre of quantum $\spl_n$ at a root of unity see \cite{T1}.
\end{remsnn}

\noindent{\it Acknowledgement}. I would like to thank S.~Donkin for mentioning the description of the regular semisimple elements in \cite[Sect.~2]{D} to me. Furthermore, I would like to thank A.~Braun, K.~A.~Brown, A.~Premet and R.~Bezrukavnikov for useful comments.

\bigskip

{\sc\noindent Department of Mathematics,
University of York,
Heslington, York, UK, YO10~5DD.
{\it E-mail address: }{\tt rht502@york.ac.uk}
}

\end{document}